\documentclass[11pt]{amsart}

\usepackage{amsmath}              
\usepackage{amsfonts}              
\usepackage{amsthm} 
\usepackage{algorithm}
\usepackage{algorithmic}
\usepackage[pdftex]{graphicx}
\usepackage{graphics}
\usepackage{mdwlist}
\usepackage{amssymb}
\usepackage{multicol}
\usepackage{pifont}
\usepackage{lscape}
\usepackage{float}
\usepackage{enumerate}

\newtheorem{thm}{Theorem}[section]
\newtheorem{lem}[thm]{Lemma}
\newtheorem{prop}[thm]{Proposition}
\newtheorem{cor}[thm]{Corollary}

\newtheorem{procedure}[thm]{Procedure}

\theoremstyle{definition}

\theoremstyle{definition}
\newtheorem{defn}[thm]{Definition}

\theoremstyle{plain}

\allowdisplaybreaks[3]  

\begin{document}

\title[The SGIS of Almost Linear Point Configurations]{The Symbolic Generic Initial System of Almost Linear Point Configurations in $\mathbb{P}^2$}
\author{Sarah Mayes}

\maketitle

\vspace*{-2em} 

\begin{abstract}
Consider an ideal $I \subseteq K[x,y,z]$ corresponding to a point configuration in $\mathbb{P}^2$ where all but one of the points lies on a single line.  In this paper we study the symbolic generic initial system $\{ \text{gin}(I^{(m)})\}_m$ obtained by taking the reverse lexicographic generic initial ideals of the uniform fat point ideals $I^{(m)}$.  We describe the \textit{limiting shape} of $\{ \text{gin}(I^{(m)})\}_m$ and, in proving this result, demonstrate that infinitely many of the ideals $I^{(m)}$ are componentwise linear.
\end{abstract}

\section{Introduction}

Given a set of distinct points $\{ p_1, \dots, p_r \}$ of $\mathbb{P}^2$, we may consider the fat point subschema $Z = m(p_1+\cdots +p_r)$ whose ideal $I_Z \subseteq K[x,y,z]$ consists of functions vanishing to at least order $m$ at each point.  If $I$ is the ideal of $\{ p_1, \dots, p_r \}$, $I_Z$ is equal to the $m$th symbolic power of $I$, $I^{(m)}$.  While uniform fat point ideals are relatively easy to describe, computing even simple invariants such as  Hilbert functions or the degree of least degree elements has proven very difficult.  Understanding how the configuration of the points $\{ p_1, \dots, p_r \}$ is related to invariants of the ideals $I^{(m)}$ is an active area of research (see, for example, \cite{CHT11}, \cite{GH07}, \cite{Markwig03}, and \cite{CH12}).

Our main objective is to describe the limiting behaviour of the Hilbert functions of the uniform fat point ideals $\{ I^{(m)}\}_m$ as $m$ gets large.  We study the case where $I$ is the ideal of a point configuration where all but one of the points lies on a single line.  The study of the asymptotic behaviour of algebraic objects has been a significant research trend over the past twenty years; it is motivated by the philosophy that the limiting behaviour of a collection of objects is often simpler than the individual elements within the collection.  For example, within the study of fat points, more can be said about the limit $\lim_{m\rightarrow \infty} \frac{\alpha(I^{(m)})}{m}$ than the individual invariants $\alpha(I^{(m)})$, where $\alpha(I^{(m)})$ denotes the degree of the least degree element of $I^{(m)}$ (for example, see \cite{Harbourne02}).

It is well-known that the Hilbert function of an ideal and its generic initial ideal are equal.  Thus, to describe the limiting behaviour of the Hilbert functions of $\{ I^{(m)}\}_m$ we will study the reverse lexicographic \textit{symbolic generic initial system} $\{ \text{gin}(I^{(m)})\}_m$ of $I$ and describe its \textit{limiting shape}.  The limiting shape $P$ of $\{ \text{gin}(I^{(m)}) \}_m$ is defined to be the limit 
$$\lim_{m \rightarrow \infty} \frac{1}{m} P_{\text{gin}(I^{(m)})}$$
where $P_{\text{gin}(I^{(m)})}$ denotes the Newton polytope of $\text{gin}(I^{(m)})$.  When $I$ is an ideal corresponding to a point configuration in $\mathbb{P}^2$ each reverse lexicographic generic initial ideal $\text{gin}(I^{(m)})$ is generated in two variables; thus $P_{\text{gin}(I^{(m)})}$, and $P$ itself, may be thought of as a subset of $\mathbb{R}^2$.  There is evidence that this limiting shape captures geometric information about the corresponding arrangement  of points (see discussion in Section 5 of \cite{Mayes6points}).

The main result of this paper is the following theorem describing the limiting shape of $\{ \text{gin}(I^{(m)})\}_m$ when $I$ is an ideal of a point configuration where all but one of the points lies on a single line.

\begin{thm}
\label{thm:limitingshapelpointson}
Fix some integer $l > 2$ and let $I \subset K[x,y,z]$ be the ideal corresponding to the arrangement of $l+1$ points $p_1, \dots, p_{l+1}$ of $\mathbb{P}^2$ such that $p_1, \dots, p_l$ lie on a line $L$ and $p_{l+1}$ does not lie $L$.  Then the limiting shape of the symbolic generic initial system  $\{ \textnormal{gin}(I^{(m)} )\}_m$ of $I$ is the shaded polytope pictured in Figure \ref{fig:limitingshape}.
\end{thm}

\begin{figure}
\begin{center}
\includegraphics[width=5cm]{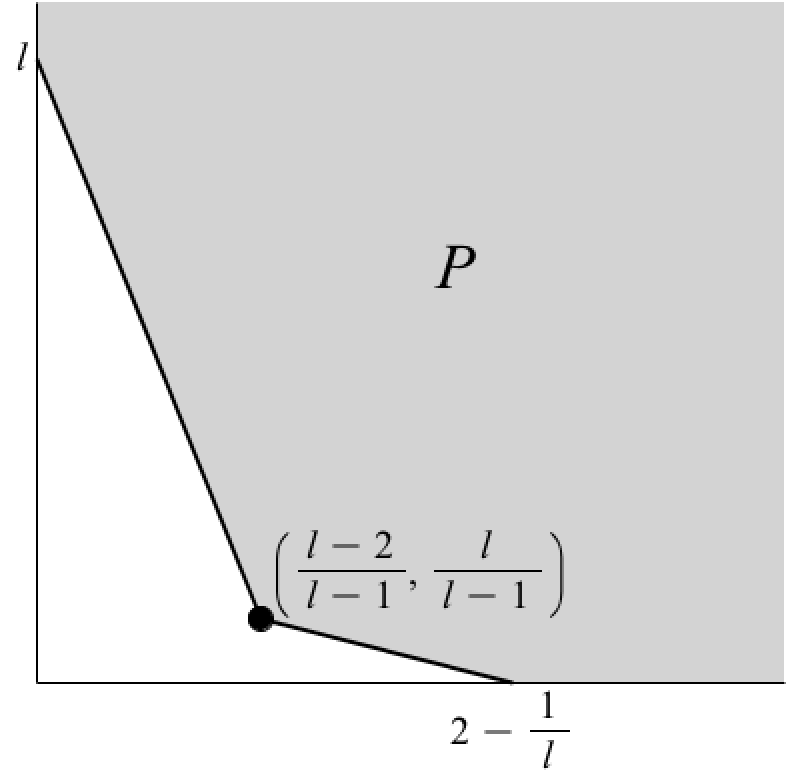}
\end{center}
\caption{The limiting shape of $\{ \textnormal{gin}(I^{(m)} )\}_m$ where $I$ is an ideal corresponding to a point configuration with $l$ points on a line and one point off of that line.}
\label{fig:limitingshape}
\end{figure}

In proving this theorem we will show that when $I$ is the ideal of such an almost linear point configuration,  $I^{(m)}$ is \textit{componentwise linear} for infinitely many $m$ (Theorem \ref{thm:componentwiselinearlpointson}).  This property means that the minimal free resolution of the ideal has a very simple form.  Other classes of ideals that are componentwise linear include stable monomial ideals, Gotzmann ideals, and ideals of at most $n+1$ fat points in general position in $\mathbb{P}^n$ (\cite{HH99}, \cite{Francisco05}).

Background information necessary for the proof of the main result is contained in Section \ref{sec:background}.  In Section \ref{sec:genscomplin} we prove results on componentwise linearity for individual fat point ideals.  Section \ref{sec:asympproof} uses these results to prove Theorem \ref{thm:limitingshapelpointson}.

\section{Background}
\label{sec:background}

In this section we will review facts about componentwise linearity, generic initial ideals of fat points, and blow-ups of points in $\mathbb{P}^2$.  Throughout $R = K[x,y,z]$ is a polynomial ring over a field of characteristic 0 with the standard grading and the reverse lexicographic order where $x > y> z$.

\subsection{Componentwise linearity}

Componentwise linear ideals are homogeneous ideals with particularly nice minimal free resolutions.

\begin{defn}[{\cite{HH99}}]
\label{defn:componentwiselinear}
Let $R = K[x_1, \dots, x_n]$ and $M$ be a graded $R$-module.  Then $M$ has a \textbf{$d$-linear resolution} if the graded minimal free resolution of $M$ is of the form
$$0 \rightarrow R(-d-s)^{\beta_s} \rightarrow \cdots \rightarrow R(-d-1)^{\beta_1} \rightarrow R(-d)^{\beta_0} \rightarrow M \rightarrow 0.$$
For any homogeneous ideal $I \subset R$, let $(I_k)$ be the ideal generated by all homogeneous polynomials of degree $k$ contained in $I$.  A homogeneous ideal $I$ is said to be \textbf{componentwise linear} if $(I_k)$ has a linear resolution for all $k$.
\end{defn}

The following theorem of Aramova, Herzog, and Hibi connects componentwise linearity to the study of generic initial ideals and will be our main tool for detecting this property.

\begin{thm}[{\cite{AHH00}}]
\label{thm:complingin}
Let $I$ be a homogeneous ideal of $K[x_1, \dots, x_n]$.  Then $I$ is componentwise linear if and only if  $I$ and its reverse lexicographic generic initial ideal $\textnormal{gin}(I)$ have the same Betti numbers.  
\end{thm}

\subsection{Generic initial ideals of fat point ideals}

When $I$ is the ideal of distinct points of $\mathbb{P}^2$, the reverse lexicographic generic initial ideals $\text{gin}(I^{(m)})$ have a very simple form detailed in the following proposition.  

\begin{prop}
\label{prop:formofgin}
Suppose that $I \subseteq K[x,y,z]$ is the ideal of a set of distinct points of $\mathbb{P}^2$.  Then the minimal generators of $\textnormal{gin}(I^{(m)})$ under the reverse lexicographic order are of the form 
$$\{ x^{\alpha(m)}, x^{\alpha(m) -1}y^{\lambda_{\alpha(m)-1}(m)}, \dots, xy^{\lambda_1(m)}, y^{\lambda_0(m)}\}$$
where $\lambda_0(m) > \lambda_1(m) > \cdots > \lambda_{\alpha(m)-1}(m) \geq 1$.
\end{prop}

This follows from the fact that generic initial ideals are Borel-fixed and  the ideals $I^{(m)}$ and $\text{gin}(I^{(m)})$ is saturated; see Corollary 2.9 of \cite{Mayesgenericpoints} for a proof.  The following corollary now follows from Theorem \ref{thm:complingin} and Proposition \ref{prop:formofgin}.

\begin{cor}
\label{for:determinedwhencomplin}
Let $I$ be the ideal of a set of distinct points in $\mathbb{P}^2$ and $m$ be an integer such that $I^{(m)}$ is componentwise linear.  The generators of $\textnormal{gin}(I^{(m)})$ are completely determined by the degrees of the minimal generators of $I^{(m)}$.
\end{cor}


\subsection{Blow-ups of Points in $\mathbb{P}^2$}
\label{sec:surfacebackground}
The algorithms that we will use to prove Theorems \ref{thm:limitingshapelpointson} and \ref{thm:componentwiselinearlpointson} come from \cite{Harbourne98}  and are very similar to the procedures outlined in \cite{Mayes6points}.   The key to these algorithms is to consider divisors on the blow-ups of the point arrangements that we are considering.  

Suppose that $\pi: X \rightarrow \mathbb{P}^2$ is the blow-up of distinct points $p_1, \dots, p_r$ of $\mathbb{P}^2$. Let $E_i = \pi^{-1}(p_i)$ for $i = 1, \dots, r$ and let $L$ be the total transform in $X$ of a line  not passing through any of the points $p_1, \dots, p_r$.  The classes of these divisors form a basis of $\text{Cl}(X)$; for convenience, we will write $e_i$ in place of $[E_i]$ and $e_0$ in place of $[L]$.  Further, the intersection product in $\text{Cl}(X)$ is defined by $e_i^2 = -1$ for $i=1, \dots, r$; $e_0^2 = 1$; and $e_i \cdot e_j = 0$ for all $i\neq j$.

Let $Z = m(p_1+\cdots +p_r)$ be a uniform fat point subscheme with sheaf of ideals $\mathcal{I}_Z$; set 
$${F}_d = dE_0 - m(E_1 + E_2 + \cdots +E_r)$$
and $\mathcal{F}_d = \mathcal{O}_X(F_d)$.  
Much of our interest in the blow-ups comes from the fact that the Hilbert function of $I^{(m)}$ is related to the divisors $F_d$ (see \cite{Mayesgenericpoints}):
$$h^0(X, \mathcal{F}_d) = H_{I^{(m)}}(d).$$
For convenience, we will sometimes write $h^0(X, F) = h^0(X, \mathcal{O}_X(F))$.  Recall that if  $[{F}]$ is not the class of an effective divisor then $h^0(X, {F}) = 0$.  On the other hand, if $F$ is effective, then we will see that we can compute $h^0(X,{F})$ by finding $h^0(X,{H})$ for some \textit{numerically effective} divisor $H$.  

\begin{defn}
A divisor $H$ is \textbf{numerically effective} if $[F] \cdot [H] \geq 0$ for every effective divisor $F$.  The cone of classes of numerically effective divisors in $\text{Cl}(X)$ is denoted NEF($X$).
\end{defn}

\begin{lem}
\label{lem:h0ofNEFF}
Suppose that $X$ is the blow-up of $\mathbb{P}^2$ at distinct points $p_1, \dots, p_r$.  Let $F \in \textnormal{NEF}(X)$.  Then $F$ is effective and 
$$h^0(X, F)  = ([F]^2-[F]\cdot [K_X])/2+1$$
where $K_X = -3E_0 + E_1 + \cdots + E_r$.
\end{lem}

\begin{proof}
This is a consequence of Riemann-Roch and the fact that $h^1(X, F) = 0$ for any numerically effective divisor $F$ on $X$. See Lemmas III.i.1(b) and II.2 of \cite{Harbourne98} for a discussion.
\end{proof}

Knowing how to compute $h^0(X, H)$ for a numerically effective divisor $H$ will allow us to compute $h^0(X, F)$ for \textit{any} divisor $F$.  In particular, given a divisor $F$, there exists a divisor $H$ such that $h^0(X, F) = h^0(X, H)$ and either: 
\begin{enumerate}
\item[(a)]  $H$ is numerically effective so $$h^0(X, F) = h^0(X, H) = ([H]^2-[H]\cdot [K_X])/2+1$$ by Lemma \ref{lem:h0ofNEFF}; or
\item[(b)]  there is a numerically effective divisor $G$ such that $[G]\cdot [H] <0$ so $[H]$ is not the class of an effective divisor and $h^0(X, F) = h^0(X, H) = 0$.
\end{enumerate}

The set of classes of effective, reduced, and irreducible curves of negative self-intersection in $X$ is used to find such an $H$; it is denoted
$$\text{NEG}(X) := \{ [C] \in \text{Cl}(X): [C]^2 < 0 , C \text{ is effective, reduced, and irreducible}\}.$$

\begin{lem}
\label{lem:Equalh0}
Suppose that $[C]  \in \textnormal{NEG}(X)$ is such that $[F] \cdot [C] <0$.  Then $h^0(X, F) = h^0(X, F-C)$.
\end{lem}

We have the following enumeration of the elements of $\text{NEG}(X)$ from Lemma III.i.1(c) of \cite{Harbourne98}.

\begin{lem}[{\cite{Harbourne98}}]
\label{lem:negelements}
Let $X$ the be blow-up of points $p_1, \dots, p_{l+1} \in \mathbb{P}^2$ where $p_1, \dots, p_l$ lie on a line and $p_{l+1}$ lies off of that line.  Then 
$$\textnormal{NEG}(X) = \{ e_0-e_1-\cdots -e_l, e_0-e_i-e_{l+1} \text{ for } i=1, \dots, l,  e_i \text{ for } i=1, \dots, l+1 \}.$$

\end{lem}

The method for finding an $H$ satisfying (a) or (b) above is as follows. 

\begin{procedure}[{Remark 2.4 of \cite{GH07}}]
\label{proc:findH}
Given a divisor $F$ we can find a divisor $H$ with $h^0(X, F) = h^0(X, H)$ satisfying either condition (a) or (b) above as follows.
\begin{enumerate}
\item Reduce to the case where $[F] \cdot e_i \geq 0$ for all $i=1, \dots, n$:  if $[F]\cdot e_i <0$ for some $i$, $h^0(X, F) = h^0(X, F-([F]\cdot e_i)E_i)$, so we can replace $F$ with $F - ([F]\cdot e_i) E_i$.
\item Since $L$ is numerically effective, if $[F]\cdot e_0<0$ then $[F]$ is not the class of an effective divisor and we can take $H=F$ (case (b)).
\item If $[F] \cdot [C] \geq 0$ for every $[C] \in \text{NEG}(X)$ then, by Lemma \ref{lem:negelements}, $F$ is numerically effective, so we may take $H = F$ (case (a)).
\item If $[F] \cdot [C] <0$ for some $[C]\in \text{NEG}(X)$ then $h^0(X, F) = h^0(X, F-C)$ by Lemma \ref{lem:Equalh0}. Replace $F$ with $F-C$ and repeat from Step 2.  
\end{enumerate}
\end{procedure}

There are only a finite number of elements in $\text{NEG}(X)$ to check by Lemma \ref{lem:negelements} so it is possible to complete Step 3.  Further, $[F] \cdot e_0 > [F-C] \cdot e_0$ when $[C] \in \text{NEG}(X)$, so the condition in Step 2 will be satisfied after at most $[F]\cdot e_0 +1$ repetitions.  Thus, the process will terminate.

Denote the number of minimal generators of $I^{(m)}$ of degree $d$ by $v_d(I^{(m)})$. Then
\begin{eqnarray*}
v_{d+1}(I^{(m)}) &=& \text{dim}(\text{coker}(  (I^{(m)})_d \otimes R_1 \rightarrow (I^{(m)})_{d+1}  ))\\
&=& \text{dim}(  \text{coker} ( H^0(X, \mathcal{F}_d) \otimes H^0(X, e_0) \rightarrow H^0(X, \mathcal{F}_{d+1}) ))\\
&:=& s(\mathcal{F}_d, e_0)
\end{eqnarray*}

If $[{F}_d]$ is not the class of an effective divisor then $h^0(X, \mathcal{F}_d)=0$ and 
\begin{eqnarray}
\label{eqn:snoteff}
s(\mathcal{F}_d, e_0) = h^0(X, \mathcal{F}_{d+1}).
\end{eqnarray}  

If $[{F}_d]$ is the class of an effective divisor let $H_d$ be the numerically effective divisor produced by Procedure \ref{proc:findH}.  Then 
$$s(\mathcal{F}_d, e_0) = s(\mathcal{H}_d, e_0) + h^0(X, \mathcal{F}_{d+1}) - h^0(X, \mathcal{H}+e_0)$$
by Lemma II.10 of \cite{Harbourne98}.  Further, since ${H}$ is numerically effective by definition, $s(\mathcal{H}_d, e_0)=0$ when the points $p_1, \dots, p_r$  lie on a conic by Theorem III.i.2 of \cite{Harbourne98}.  Thus, in the cases we are interested in, 
\begin{eqnarray}
\label{eqn:seff}
s(\mathcal{F}_d, e_0) = h^0(X, \mathcal{F}_{d+1}) - h^0(X, \mathcal{H}+e_0).
\end{eqnarray}

Therefore, to find the number of generators $s(\mathcal{F}_d, e_0)$ of each degree $d+1$, we will proceed as follows.\footnote{Other algorithms for determining the number of generators are possible for the cases we are considering.  For example,  the techniques from \cite{CHT11} yield the same results.}
\begin{enumerate}[(a)]
\item Follow Procedure \ref{proc:findH} to determine if ${F}_d$ is effective or non-effective.  If ${F}_d$ is effective, the procedure will yield a numerically effective divisor ${H}_d$. 
\item Compute $v_{d+1}(I^{(m)}) = s(\mathcal{F}_d, e_0)$ for each $d$ using expressions (\ref{eqn:snoteff}) or (\ref{eqn:seff}) together with Procedure \ref{proc:findH} and the formula from Lemma \ref{lem:h0ofNEFF}.
\end{enumerate}

\section{Generators of $I^{(m)}$ and Componentwise Linearity}
\label{sec:genscomplin}

Throughout this section $R = K[x,y,z]$ is a polynomial ring over a field of characteristic zero with the reverse lexicographic order and $I$ is the ideal of a point configuration $\{ p_1, \dots, p_{l+1} \}$ where $p_1, \dots, p_l$ lie on a single line and $p_{l+1}$ is off of that line.  The purpose of this section is to enumerate the generators of the fat point ideals $I^{(m)}$ and, in doing so, prove the following theorem.

\begin{thm}
\label{thm:componentwiselinearlpointson}
Let $I \subseteq K[x,y,z]$ be an ideal of $l+1$ points where $l$ points  lie on a single line.  Then an infinite number of the uniform fat point ideals $I^{(m)}$ are componentwise linear.  In particular, $I^{(m)}$ is componentwise linear when $l(l-1)$ divides $m$.
\end{thm}

The following proposition gives a specific criterion for an ideal of fat points to be componentwise linear.

\begin{prop}
\label{prop:componentwiseequivalence}
Let $J$ be a homogeneous ideal of $K[x_1, \dots, x_n]$ such that the reverse lexicographic generic initial ideal $\textnormal{gin}(J)$ is generated in two variables.  If $\alpha$ is the degree of the smallest degree generator of $J$ and
$$\alpha = \{ \text{number of minimal generators of } J \}-1$$
then $J$ is componentwise linear.
\end{prop}

\begin{proof}
By Theorem \ref{thm:complingin}, $J$ is componentwise linear if and only if $J$ and $\text{gin}(J)$ have the same Betti numbers.  Since the Betti numbers of $J$ are obtained from those of $\text{gin}(J)$ by making a series of consecutive cancellations (see Section I.22 of \cite{Peeva11}),  $J$ is componentwise linear if and only if no consecutive cancellations occur.  However, since the minimal free resolution of $\text{gin}(J)$ is of the form
$$0 \rightarrow \bigoplus_j R(-j)^{\beta_{1,j}} \rightarrow  \bigoplus_j R(-j)^{\beta_{0,j}} \rightarrow \text{gin}(J) \rightarrow 0,$$
any consecutive cancellation must involve cancelling a $\beta_{0,j}$; these Betti numbers correspond to minimal generators of $\text{gin}(J)$.  Therefore, showing that $J$ is componentwise linear in this case is equivalent to showing that the minimal generators of $J$ and $\text{gin}(J)$ are of the same degrees or, equivalently by consecutive cancellation, that $J$ and $\text{gin}(J)$ have the same number of generators.

Since $\alpha$ is the degree of the least degree generator of $J$, it is also the degree of the least degree generator of  $\text{gin}(J)$.  By Borel-fixedness,
$$\text{gin}(J) = (x_1^{\alpha}, x_1^{\alpha-1}x_2^{\lambda_{\alpha-1}}, \dots, x_1x_2^{\lambda_1}, x_2^{\lambda_0})$$
for some invariants $\{ \lambda_i\}_i$ and $\text{gin}(J)$ has $\alpha+1$ generators.  Since $J$ also has $\alpha+1$ generators, it must be componentwise linear.
\end{proof}

To prove Theorem \ref{thm:componentwiselinearlpointson} it remains to show that the conditions of the Proposition \ref{prop:componentwiseequivalence} are satisfied when $J=I^{(m)}$ and $l(l-1)$ divides $m$.  That is, we need to show that  the degree of the smallest degree generator of $I^{(m)}$ is one less than the number of minimal generators of $I^{(m)}$.  To demonstrate this we will compute the number of generators of each degree using the procedure outlined in Section \ref{sec:surfacebackground}.


\subsection{Finding ${H}_d$}
\label{sec:Zardecomp}
Fix points $p_1, \dots, p_l$ be points of $\mathbb{P}^2$ lying on a line and let $p_{l+1}$ be a point off of that line.  Let $I$ be the ideal of $\{ p_1, \dots, p_{l+1} \} $ and $X$ be the blow-up of the points $p_1, \dots, p_{l+1}$.
Throughout this section we will assume that $m=\rho l (l-1)$ for some $\rho \in \mathbb{N}$ and write 
$$a_0E_0 + a_1E_1 +\cdots+ a_rE_{l+1} := (a_0; a_1, \dots, a_{l+1}).$$
For convenience we will write elements of $\text{NEG}(X)$ as
$$A:=E_0- E_1- \cdots - E_l, \qquad B_i:=E_0-E_i-E_{l+1}.$$

\subsubsection{$d \geq lm$}

In this case $[F_d] \cdot [C] \geq 0$ for all $[C] \in \text{NEG}(X)$ so
$$\mathcal{H}_d = \mathcal{F}_d.$$

\subsubsection{$ 2m \leq d < lm$}
\label{sec:misc2}

In this case we may subtract copies of $A$ in Procedure \ref{proc:findH}, but no copies of the $B_i$s are subtracted.
\begin{eqnarray*}
{H}_d &=&{F}_d - \Big \lceil  \frac{lm-d}{l-1} \Big \rceil A\\
&=& \Big ( d- \Big \lceil  \frac{lm-d}{l-1} \Big \rceil ; m - \Big \lceil  \frac{lm-d}{l-1} \Big \rceil , \dots, m- \Big \lceil  \frac{lm-d}{l-1} \Big \rceil  ,m    \Big) 
\end{eqnarray*}

\subsubsection{$2m - \frac{m}{l} \leq d < 2m$}
\label{sec:misc1}

Write $d = 2m-\gamma$.  Both $[F_d] \cdot [B_i] < 0$ and $[F_d] \cdot [A] < 0$, so we may subtract copies of $A$ and $B_i$.  Procedure \ref{proc:findH} yields
\begin{eqnarray*}
{H}_d &=&{F}_d -  \Big \lceil  \frac{lm-d}{l-1} \Big \rceil A - \gamma(B_1+\cdots + B_l) \\
&=& \Big( d-  \Big \lceil  \frac{lm-d}{l-1} \Big \rceil - \gamma l ; m -  \Big \lceil  \frac{lm-d}{l-1} \Big \rceil - \gamma, \dots, m-  \Big \lceil  \frac{lm-d}{l-1} \Big \rceil- \gamma, m-\gamma l \Big )
\end{eqnarray*}

\subsubsection{$d< 2m-\frac{m}{l}$}
In this case Procedure \ref{proc:findH} will eventually yield a divisor class $[G] = a_0e_0 - \cdots - a_re_r$ where $a_0 <0$ so $G$, and thus $F_d$, is not effective.  In this case
$h^0(X, \mathcal{F}_d) = 0.$

\subsection{Determining $s(\mathcal{F}_d, e_0)$}
\label{sec:sFd}
Fix $I$, $X$, and $m = \rho l(l-1)$ as in the previous section.  
We may compute $s(\mathcal{F}_d, e_0)$ using expression (\ref{eqn:snoteff}) when $\mathcal{F}_d$ is not effective and expression (\ref{eqn:seff}) when $\mathcal{F}_d$ is effective.  We will use the following information to evaluate these expressions.
\begin{itemize}
\item The divisors ${H}_d$ and ${H}_{d+1}$ computed in the previous section.
\item When $[{F}_{d+1}]$ is the class of an effective divisor,
$$h^0(X, \mathcal{F}_{d+1}) = h^0(X, \mathcal{H}_{d+1}) = ([{H}_{d+1}]^2 - [K_X] \cdot [{H}_{d+1}]) /2 +1$$ 
by Lemma \ref{lem:h0ofNEFF} and Procedure \ref{proc:findH}.
\end{itemize}

\subsubsection{$d \geq lm$}

In this case, ${H}_d + E_0 ={F}+E_0 ={F}_{d+1}$ so 
$$s(\mathcal{F}_d, e_0) = h^0(X, \mathcal{F}_{d+1}) - h^0(X, \mathcal{F}_{d+1}) = 0$$

\subsubsection{$2m \leq d \leq lm-2$}

Write $d = jm+w(l-1)+p$ where $j = 2, \dots, l-1$, $w = 0,\dots, \rho l -1$, and $p = 0, \dots, l-2$.  Then 
$ \Big \lceil  \frac{lm-d}{l-1} \Big \rceil = (l-j) l \rho -w$ and, referring to the expression for ${H}_d$ from Section \ref{sec:misc2}, we see that ${H}_d+E_0$ and ${H}_{d+1}$ will be equal when $p \neq l-2$.  Thus, when $p \neq l-2$, 
$$s(\mathcal{F}_d, e_0) = h^0(X, \mathcal{H}_{d+1}) - h^0(X, \mathcal{H}_d + e_0)=0.$$
When $p = l-2$, we compute
$$s(\mathcal{F}_d, e_0) = h^0(X, \mathcal{F}_{d+1}) - h^0(X, \mathcal{H}_d+e_0) = 1.$$

\subsubsection{$2m-\frac{m}{l} \leq d < 2m$}

Write $d = 2m - (p+w(l-1))$ where and: $w=0, \dots, \rho$, $p= 1, \dots, l-2$ when $w = 0$; $p = 0, \dots, l-2$ when $w= 1, \dots, \rho-1$; and $p=0$ when $w=\rho$.  Then we use the expressions for ${H}_d$ from Section \ref{sec:misc1} to find 
\begin{eqnarray*}
s(\mathcal{F}_d, e_0) = \begin{cases} l &\mbox{if } p\neq 1 \\ 
l+1 & \mbox{if } p=1. \end{cases} 
\end{eqnarray*}

\subsubsection{$d = 2m - \frac{m}{l} - 1$}
In this case $[{F}_d]$ is not in the class of an effective divisor so
$$s(\mathcal{F}_d, e_0) = h^0(X, \mathcal{F}_{d+1}) = h^0(X, \mathcal{H}_{d+1}) = 1.$$

\subsubsection{$d< 2m-\frac{m}{l}-1$}

In this case neither $[{F}_d]$ nor $[{F}_{d+1}]$ is in the class of an effective divisor so
$$s(\mathcal{F}_d, e_0) = h^0(X, \mathcal{F}_{d+1}) = 0.$$

\subsection{Generators of $I^{(m)}$}
\label{sec:gensofIm}
Let $I$, $X$, and $m$ be as in Sections \ref{sec:Zardecomp} and \ref{sec:sFd}.
We summarize the number of generators of each degree, using the results of the previous section and the fact that $v_{d+1}(I^{(m)}) = s(\mathcal{F}_d, e_0)$.

\subsubsection{$2m < d \leq lm$}
There is one generator of $I^{(m)}$ of degree $jm+w(l-1)+(l-2)+1 = jm+w(l-1) + l-1$ for each $j=2, \dots, l-1$, $w = 0, \dots \rho l -1$.  

\subsubsection{$2m-\frac{m}{l} +1 \leq d \leq 2m$}
\begin{itemize}
\item There are $l$ generators of degrees
$$2m-(p+w(l-1))+1$$
when $w=0$ and $p=2, \dots, l-2$ or $w = 1, \dots, \rho-1$ and $p = 0, 2, 3, \dots, l-2$.
\item There are $l$ generators of degree
$$2m-(\rho(l-1))+1 = 2m - \frac{m}{l}+1.$$
\item There are $l+1$ generators of degrees 
$$2m - (1+w(l-1))+1 = 2m-w(l-1)$$
where $w=0, 1, \dots, \rho-1$.
\end{itemize}

\subsubsection{$d = 2m-\frac{m}{l}$}
There is exactly one generator of degree $2m-\frac{m}{l}$.

\subsection{Componentwise Linearity}
\label{sec:proofofcomplin}

\begin{proof}[Proof of Theorem \ref{thm:componentwiselinearlpointson}]

Let $I$ be the ideal corresponding to a point configuration where $l$ points lie on a line and one point lies off of the line. Fix $m = \rho(l)(l-1)$ for $\rho \in \mathbb{N}$.   By Proposition \ref{prop:componentwiseequivalence} it is sufficient to show that the degree of the smallest degree generator of $I^{(m)}$ is one less than the number of elements in a minimal generating set of $I^{(m)}$.  
By our work in Section \ref{sec:gensofIm} the smallest degree generator of $I^{(m)}$ is of degree $2m-\frac{m}{l}$.  The number of minimal generators is equal to 

\begin{eqnarray*}
&& [\text{no. gens. of degree } > 2m ] + [\text{no. gens. of degree }  d, \\
& & \quad 2m-\frac{m}{l}+1 \leq d \leq 2m] +  [\text{no. gens. of degree } \leq 2m - \frac{m}{l}]  \\
&=& [(l-2)(\rho l)] + [ l(l-2)(\rho-1)+l(l-3)+l+(l+1)\rho]+1\\
&=& [ \rho l^2 - 2 \rho l] + [ \rho l^2- \rho l + \rho ] +1\\
&=& 2m - \rho l + \rho +1 = 2m - \frac{m}{l} +1.
\end{eqnarray*}

 \end{proof}
\section{Computation of the Limiting Shape}
\label{sec:asympproof}

In this section we will prove Theorem \ref{thm:limitingshapelpointson} using the fact from Theorem \ref{thm:componentwiselinearlpointson} that infinitely many of the $I^{(m)}$ are componentwise linear.

\begin{proof}[Proof of Theorem \ref{thm:limitingshapelpointson}] 
Let $m$ is divisible by $l$ and $l-1$ so that $m = \rho l (l-1)$ for some $\rho \in \mathbb{N}$.  Then $I^{(m)}$ is componentwise linear by Theorem \ref{thm:componentwiselinearlpointson}.
 Proposition \ref{prop:formofgin} then implies that 
$$\text{gin}(I^{(m)}) = (\{ x^{\alpha(m)}, x^{\alpha(m) -1}y^{\lambda_{\alpha(m)-1}(m)}, \dots, xy^{\lambda_1(m)}, y^{\lambda_0(m)}\})$$
where $\alpha(m) = 2m-\frac{m}{l}$ and $\lambda_0(m) = lm$ by our work in Section \ref{sec:gensofIm}.

We may think of the sequence of invariants $\{ \lambda_i(m)\}_i$ of $\text{gin}(I^{(m)})$ as having two phases: the first phase corresponds to $\lambda_i(m)$ where $\lambda_i(m)+i > 2m$ and the second corresponds to $\lambda_i(m)$ where $\lambda_i(m) +i \leq 2m$.  The $\lambda_i(m)$ within each of these two phases are regularly spaced; that is, there are patterns in their gaps.  Taking this into consideration, it is not difficult to see that the Newton polytope of $\text{gin}(I^{(m)})$ is defined by the points $(\alpha(m),0) = (2m-\frac{m}{l},0)$, $(J, \lambda_J(m))$ where $\lambda_J+J= 2m$, and $(0, \lambda_0(m)) = (0,lm)$ (see Figure \ref{fig:newtonpolytope}).  Also note that when $\lambda_i+i > 2m$
$$(l-1) = (\lambda_{i-1}+i-1) - (\lambda_i+i) = \lambda_{i-1}- \lambda_i - 1$$
so $\lambda_{i-1} - \lambda_i = l$ and the slope of the line $L_1$ in Figure \ref{fig:newtonpolytope}  is equal to $-l$.
There are a total of $\rho l^2 - 2 \rho l = m- \frac{m}{l-1}$ generators corresponding to the first subset of invariants so $J = m - \frac{m}{l-1}$ and 
$$\lambda_J(m) = lm - l \Big(m- \frac{m}{l-1} \Big ) = l( \frac{m}{l-1}).$$
Note that 
$$J+\lambda_J(m) = m - \frac{m}{l-1} + l \Big( \frac{m}{l-1} \Big) = m - \frac{m}{l-1} (l-1) = 2m$$
as required.

\begin{figure}
\begin{center}
\includegraphics[width=5cm]{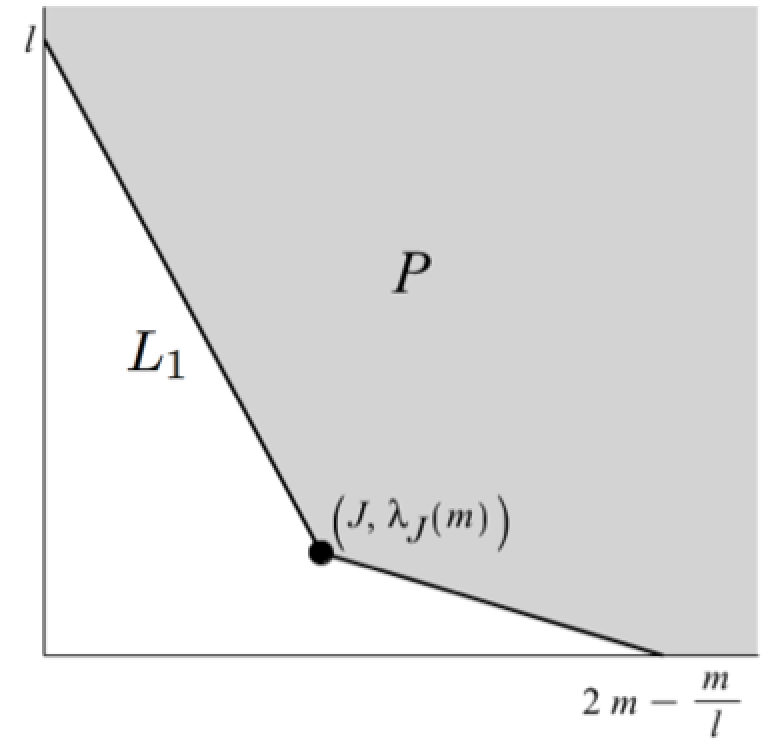}
\end{center}
\caption{The Newton polytope of $\textnormal{gin}(I^{(m)} )$ where $I$ is an ideal corresponding to a point configuration with $l$ points on a line and one point off of that line.}
\label{fig:newtonpolytope}
\end{figure}

Thus, the limiting shape of $\{ \text{gin}(I^{(m)}) \}_m$ is defined by the points:
$$\Big( \lim_{m \rightarrow \infty} \frac{2m-m/l}{m}, 0 \Big ) = \Big ( 2- \frac{1}{l}, 0\Big)$$
$$\Big( 0, \lim_{m\rightarrow \infty} \frac{lm}{m} \Big) = (0, l)$$
$$\Big (  \lim_{m\rightarrow \infty} \frac{m- \frac{m}{l-1}}{m}, \lim_{m \rightarrow \infty} \frac{l\frac{m}{l-1}}{m}\Big) = \Big ( 1- \frac{1}{l-1}, \frac{l}{l-1} \Big )$$
and claimed.

\end{proof}

One can easily check that the area under the limiting shape is equal to $\frac{l+1}{2}$.  This is consistent with the general fact that the area under the limiting shape of $\{ \text{gin}(I^{(m)}) \}_m$ when $I$ is the ideal of $r$ points is equal to $\frac{r}{2}$ (see Proposition 2.14 of \cite{Mayesgenericpoints}).

\bibliography{lpointsonbib}
\bibliographystyle{amsalpha}
\nocite{*}


\end{document}